\documentclass[11pt,reqno]{amsart}
\usepackage[T1]{fontenc}
\usepackage{lmodern}
\usepackage{microtype}
\usepackage[a4paper,margin=27mm,headheight=14pt]{geometry}
\usepackage{amsmath,amssymb,mathtools,booktabs,array}
\usepackage{needspace}
\usepackage{xcolor}
\usepackage[colorlinks=true,linkcolor=blue!45!black,citecolor=blue!45!black,urlcolor=blue!45!black]{hyperref}
\hypersetup{pdftitle={Complex Curvature and Schatten Embeddings},pdfauthor={}}
\newtheorem{theorem}{Theorem}[section]
\newtheorem{proposition}[theorem]{Proposition}
\newtheorem{corollary}[theorem]{Corollary}
\newtheorem{lemma}[theorem]{Lemma}
\theoremstyle{definition}
\newtheorem{definition}[theorem]{Definition}
\theoremstyle{remark}
\newtheorem{remark}[theorem]{Remark}
\newcommand{\C}{\mathbb C}
\newcommand{\R}{\mathbb R}
\newcommand{\N}{\mathbb N}
\newcommand{\HH}{\mathcal H}
\newcommand{\KK}{\mathcal K}
\newcommand{\Sp}{\mathcal S}
\newcommand{\Tr}{\operatorname{Tr}}
\newcommand{\rank}{\operatorname{rank}}
\newcommand{\op}{\mathrm{op}}

\newcommand{\dd}{\,d}
\newcommand{\norm}[1]{\left\lVert#1\right\rVert}
\newcommand{\mean}[1]{\frac{1}{2\pi}\int_0^{2\pi}#1\dd\theta}

\numberwithin{equation}{section}
\allowdisplaybreaks[1]
\title[Further Results on the Isometric Embeddability of \(S_q^m\) into \(S_p^n\)]{Further Results on the Isometric Embeddability of \(S_q^m\) into \(S_p^n\)}
\author{}
\date{September 26, 2026}
\subjclass[MSC.]{46B04, 46L51}

\begin{document}

\maketitle

\begin{center}
	{\large Ying Xu$^{1}$, Wenwen Zhang$^{1}$ and Qi Liu$^{1*}$}\par
	\vspace{0.8em}
	
	$^{1}$School of Mathematics and Statistics, Anqing Normal University,
	Anqing 246133, Anhui, China.\par
	
	\vspace{1.0em}
	
	\textit{Corresponding author(s).} E-mail(s):
	\href{mailto:liuq67@aqnu.edu.cn}{liuq67@aqnu.edu.cn};\par
	
	\textit{Contributing authors:}
	\href{mailto:Y26220046@stu.aqnu.edu.cn}{Y26220046@stu.aqnu.edu.cn};
	\href{mailto:Y26220004@stu.aqnu.edu.cn}{Y26220004@stu.aqnu.edu.cn}.
	\par
\end{center}

\vspace{1.2em}
	
\begin{abstract}
We study Schatten embeddings through circular means of trace powers. In finite dimensions, a decomposition by singular-value vanishing orders shows that every term of order at most two has a nonnegative coefficient. This yields the classification $\ell_q^2(\C)\hookrightarrow\Sp_p^n$ if and only if $q=p$ or $q=2$, for $n\ge2$ and $p<\infty$. In particular, it gives negative answers to parts (ii), (iii), and (iv) of Problem 3.1 in Chattopadhyay--Pradhan--Skripka, arXiv:2603.07359v2. A regularized curvature estimate also excludes $\ell_q^2$ from infinite-dimensional $\Sp_p$ for $0<p\le2<q\le\infty$ and gives finite-rank distortion bounds. The other parts of that problem are not claimed to be solved.Our results extend previous studies on Schatten embeddings, particularly those in
\cite{CHPPR,
	CHPR,
CPS},
by developing a complex-curvature framework based on complex convexity
theories \cite{BR,DGT},
which resolves several remaining quasi-Banach cases and yields new geometric
obstructions for isometric embeddings.
\end{abstract}

\maketitle

\section{Introduction and principal result}

An isometric embedding preserves the entire norm geometry of a space, 
including the local behavior of its unit sphere along two-dimensional sections. 
For matrix norms, this local behavior records more than the singular values of 
individual matrices: it also detects how singular values change under 
perturbations. The problem of embedding one Schatten class into another 
therefore brings together Banach-space geometry, matrix perturbation theory, 
and the analysis of trace functions. The distinction between real and complex 
perturbations is particularly useful in this setting. A norm can be flat along 
a real line while having strictly positive average growth around every 
sufficiently small complex circle. The theory of isometric embeddings has a 
long history in Banach space geometry. Classical results of Banach and 
Lamperti characterized isometries on $\ell_p$ and $L_p$ spaces, respectively, 
in the non-Hilbertian setting \cite{Banach,Lamperti}. The study of embeddings 
between different $L_p$ spaces has further established deep connections with 
probability theory, harmonic analysis, and the geometry of Banach spaces. In 
particular, the works of L\'evy and Schoenberg revealed fundamental connections 
between positive definite functions, stable random variables, and isometric 
embeddings into $L_p$ spaces \cite{Levy,Schoenberg}.

Structural properties of subspaces of $L_p$ spaces were investigated by 
Rosenthal and later by Delbaen, Jarchow, and Pe\l czy\'nski, who characterized 
subspaces admitting isometric embeddings into $\ell_p$ or finite-dimensional 
$\ell_p$ spaces \cite{Rosenthal,DelbaenJarchowPelczynski}. Finite-dimensional 
$\ell_p$ embeddings have also been extensively studied, where strong 
restrictions on the parameters and quantitative dimension estimates were 
obtained, with close connections to spherical designs and combinatorial 
structures \cite{Lyubich2008,Lyubich2009}. With the development of 
non-commutative analysis, these classical embedding problems have been extended 
to non-commutative $L_p$ spaces and Schatten classes, where new geometric 
phenomena arise from the underlying operator structure \cite{Yeadon,JungeParcet,Xu}. 
These developments have led to the recent study of isometric embeddings between 
Schatten classes, connecting Banach space geometry with operator theory and 
non-commutative $L_p$ geometry \cite{CHPPR,CHPR}.

For a complex Hilbert space $\HH$ and $0<p<\infty$, write
\[
 \Sp_p(\HH)=\left\{A\text{ compact}:\sum_{j\ge1}s_j(A)^p<\infty\right\},
 \qquad \norm{A}_p^p=\Tr|A|^p.
\]
Here $s_j(A)$ are the singular values, counted with multiplicity. We set $\Sp_\infty(\HH)=\mathcal K(\HH)$ and use the operator norm $\norm{\cdot}_\op$. For $\HH=\C^n$ we write $\Sp_p^n$. The Hilbert--Schmidt norm is $\norm{\cdot}_2$. All isometries and embeddings below are \emph{complex-linear}, unless explicitly stated otherwise. In particular, $\ell_q^2$ means $\ell_q^2(\C)$.

Chattopadhyay, Hong, Pal, Pradhan, and Ray studied finite-dimensional
Schatten embeddings using analytic perturbation and trace differentiation
\cite{CHPPR}. The quasi-Banach range was subsequently investigated by
Chattopadhyay, Hong, Pradhan, and Ray \cite{CHPR}. Their non-embeddability
results cover many parameters, but leave several cases untreated. The survey
of Chattopadhyay, Pradhan, and Skripka \cite{CPS} collected the remaining
questions in Problem 3.1 concerning the existence of isometric embeddings
between Schatten classes. In particular, the unresolved cases include(the original numbering of Problem~3.1 in \cite{CPS} is retained):

(ii) Do $\Sp_3^m$ embeds isometrically into $\Sp_1^n$;

(iii) Do $\Sp_1^m$ and $\Sp_\infty^m$ embed isometrically into
$\Sp_p^n$ for
$p\in\{\frac{1}{k}:k\in\mathbb N\setminus\{1\}\}$;

(iv) Does $\Sp_q^m$ embeds isometrically into $\Sp_p^n$ for
$q\in(2,\infty)$ and $p\in(0,1)$.

The present work addresses these cases by developing a complex-curvature
framework for Schatten embeddings. Combining circular mean estimates with
finite-dimensional holomorphic expansions, we obtain new non-embeddability
criteria and resolve the above problems.

The relevant geometric background is classical. Davis, Garling, and Tomczak-Jaegermann developed uniform complex convexity for quasi-normed spaces \cite{DGT}. Blower and Ransford related complex convexity to Riesz measures of trace functions and proved quantitative estimates for $1\le p\le2$ \cite{BR}. The present argument uses the same broad principle---positive complex curvature---but retains a Hilbert--Schmidt term. This weaker choice of control permits a direct calculation throughout $0<p<1$, without asserting uniform control by $\norm{B}_p^2$.

\Needspace{9\baselineskip}
\begin{theorem}\label{thm:main}
Let $0<p\le2$ and $2<q\le\infty$. For every complex Hilbert space $\KK$, there is no complex-linear isometric embedding
\[
 \ell_q^2\longrightarrow\Sp_p(\KK).
\]
Consequently, for every $m\ge2$ and every finite $n$, there is no complex-linear isometric embedding $\Sp_q^m\to\Sp_p^n$. More generally, if $\dim\HH\ge2$, there is no such embedding $\Sp_q(\HH)\to\Sp_p(\KK)$.
\end{theorem}

\noindent
\textnormal{This gives complete negative answers to Problem~3.1(ii) and (iv), and a partial negative answer to Problem~3.1(v) in \cite{CPS}.}

The theorem has no positivity, self-adjointness, rank-preservation, or finite-rank hypothesis on the images. The assertion for Schatten spaces follows by restricting an alleged isometry to the diagonal span of two orthogonal rank-one projections. The actual obstruction is therefore already two-dimensional over $\C$.

	A recent result of Hein\"avaara \cite{HeinavaaraTracial} gives a complementary
	description of two-dimensional Schatten geometry: for every $p>0$, each real
	two-dimensional subspace of $S_p$ is linearly isometric to a subspace of a
	commutative $L_p$ space. This concerns real linear combinations of two
	operators, whereas the present work studies complex-linear embeddings and the
	full complex pencil $A+zB$, $z\in\mathbb C$. Thus, Hein\"avaara's result does
	not settle the embedding problems considered here, but provides a useful
	two-dimensional comparison framework.

In finite dimensions, a more precise expansion treats the reciprocal exponents as well. The connection between vanishing orders and singular values is described by the local Smith form; see, for example, Kaveh and Makhnatch \cite{KM}. {We use this standard form and record only the singular-value consequence needed below, combining it with analytic spectral projections and angular averaging.}

\begin{theorem} \label{thm:finiteintro}
	
Let $n\ge2$, $0<p<\infty$, and $0<q\le\infty$. A complex-linear isometric embedding $\ell_q^2(\C)\to\Sp_p^n$ exists if and only if $q=p$ or $q=2$. Consequently, if $m\ge2$, $p\ne q$, and $q\ne2$, there is no complex-linear isometric embedding $\Sp_q^m\to\Sp_p^n$.
\end{theorem}

\noindent
\textnormal{In particular, this yields a complete negative answer to Problem~3.1(iii) in \cite{CPS}.}

This theorem concerns the admissible exponents; it does not determine the least target size for a prescribed higher-dimensional Hilbertian source. Its proof, in Section~\ref{sec:expansion}, uses a property not contained in a quadratic lower bound alone: the low-order circular expansion has no negative coefficient. In particular, for $p<1$ it cannot coincide with $(1+t)^p$, whose quadratic coefficient is negative.

The quantitative ingredient is the following inequality, proved in Section~\ref{sec:circles}. If $A,B\in\Sp_p(\KK)$, $B\ne0$, and $t>0$, then
\begin{equation}\label{eq:circleintro}
 \mean{\norm{A+t e^{i\theta}B}_p^p}-\norm{A}_p^p
 \ge \frac{p^2}{4}\,
 \frac{t^2\norm{B}_2^2}
 {(\norm{A}_\op+t\norm{B}_\op)^{2-p}}.
\end{equation}
Because $p\le2$, every member of $\Sp_p$ is Hilbert--Schmidt, so the expression is finite. For a fixed nonzero $A$ and nonzero $B$, the right-hand side has a strictly positive quadratic coefficient at $t=0$. An isometric copy of $\ell_q^2$, on the other hand, would make the left-hand side equal to
\[
 (1+t^q)^{p/q}-1=O(t^q)=o(t^2),\qquad 2<q<\infty,
\]
or identically zero for $q=\infty$ and $0<t<1$.

\paragraph{Scope of the answers.}\label{sec:scope}
Theorems~\ref{thm:main} and \ref{thm:finiteintro} give complete negative answers to Problem 3.1(ii), (iii), and (iv). They also exclude the ranges $2<q\le\infty$, $0<p<1$, and $2<q<4$, $p=1$, in (v). They do not settle (i), (vi), or the remaining portions of (v). These distinctions are recorded in Section~\ref{sec:scope}.

There is also an essential distinction between a consequence and a new geometric principle. At $p=1$, Theorem~\ref{thm:main} follows directly from the classical $2$-uniform PL-convexity of the trace class. Thus the resolution of (ii) is an application of established theory, rather than a claim that complex convexity is new. The self-contained regularization argument below provides a unified proof, including the quasi-Banach range, and explicit rank-dependent bounds. No classification of all Schatten embeddings is asserted.

\section{A regularized trace Hessian}\label{sec:hessian}

For $z=x+iy$, we use the standard Wirtinger derivatives
\[
\partial_z=\frac12(\partial_x-i\partial_y),
\qquad
\partial_{\bar z}=\frac12(\partial_x+i\partial_y),
\qquad
\Delta=4\partial_z\partial_{\bar z}.
\]
For the matrix-function differentiation used below, we refer to
\cite{Higham2008}.

For a differentiable scalar function $h$ on $(0,\infty)$, put
\[
 h^{[1]}(a,b)=
 \begin{cases}
 \dfrac{h(a)-h(b)}{a-b},&a\ne b,\\[6pt]
 h'(a),&a=b.
 \end{cases}
\]
The use of positive regularization is important: differentiating $\Tr|X|^p$ directly at a singular matrix is generally invalid, particularly for $p<1$.

\begin{lemma}\label{lem:trace}
Let $H(z)$ be a twice continuously differentiable Hermitian matrix-valued function near $z=0$. Suppose that $H(0)=\operatorname{diag}(\lambda_1,\ldots,\lambda_n)$ and let $f$ be analytic on a complex neighborhood of its spectrum; the neighborhood may be disconnected. Then
\begin{align}\label{eq:tracehessian}
 \left.\partial_z\partial_{\bar z}\Tr f(H(z))\right|_{z=0}
 &=\Tr\bigl(f'(H(0))H_{z\bar z}(0)\bigr)\notag\\
 &\quad+\sum_{i,j=1}^n(f')^{[1]}(\lambda_i,\lambda_j)
 (H_z(0))_{ij}(H_{\bar z}(0))_{ji}.
\end{align}
Repeated eigenvalues are permitted.
\end{lemma}

\begin{proof}
{
By standard analytic functional calculus,
\[
 \partial_z\Tr f(H(z))=\Tr\bigl(f'(H(z))H_z(z)\bigr),
\]
and, at a diagonal matrix $H(0)$, the Fr\'echet derivative of $f'$ satisfies
\[
 \bigl(Df'(H(0))[E]\bigr)_{ij}
 =(f')^{[1]}(\lambda_i,\lambda_j)E_{ij},
\]
with the diagonal case interpreted by $h^{[1]}(a,a)=h'(a)$. Applying the product rule and cyclicity of the trace gives \eqref{eq:tracehessian}.
}
\end{proof}

\begin{proposition}\label{prop:levi}
Let $0<p\le2$, $r=p/2$, and $\varepsilon>0$. For $X,B\in M_n(\C)$, define
\[
 F_\varepsilon(X)=\Tr(X^*X+\varepsilon I)^r.
\]
Then
\begin{equation}\label{eq:levilower}
 \left.\partial_z\partial_{\bar z}F_\varepsilon(X+zB)\right|_{z=0}
 \ge r^2(\norm{X}_\op^2+\varepsilon)^{r-1}\norm{B}_2^2.
\end{equation}
In particular, on an operator-norm ball $\norm{X}_\op\le R$,
\begin{equation}\label{eq:balllevi}
 \left.\partial_z\partial_{\bar z}F_\varepsilon(X+zB)\right|_{z=0}
 \ge r^2(R^2+\varepsilon)^{r-1}\norm{B}_2^2.
\end{equation}
\end{proposition}

\begin{proof}
Take a singular value decomposition $X=UDV^*$, where $D=\operatorname{diag}(s_1,\ldots,s_n)$ with $s_i\ge0$, and put $C=U^*BV=(c_{ij})$. Two-sided unitary invariance reduces the calculation to $D+zC$. Set
\[
 H(z)=(D+zC)^*(D+zC)+\varepsilon I,
 \qquad \lambda_i=s_i^2+\varepsilon.
\]
At zero,
\[
 H_z=DC,\qquad H_{\bar z}=C^*D,\qquad H_{z\bar z}=C^*C.
\]
Lemma~\ref{lem:trace}, applied to $f(u)=u^r$, gives the exact formula
\begin{equation}\label{eq:exactcoeff}
 \left.\partial_z\partial_{\bar z}F_\varepsilon(D+zC)\right|_{z=0}
 =\sum_{i,j=1}^n\gamma_{ij}|c_{ij}|^2,
\end{equation}
where, with continuous interpretation when $\lambda_i=\lambda_j$,
\begin{align}\label{eq:coefficient}
 \gamma_{ij}
 &=r\left[\lambda_j^{r-1}+(\lambda_i-\varepsilon)
 \frac{\lambda_i^{r-1}-\lambda_j^{r-1}}{\lambda_i-\lambda_j}\right]\notag\\
 &=r\left[\frac{\lambda_i^r-\lambda_j^r}{\lambda_i-\lambda_j}
 -\varepsilon\frac{\lambda_i^{r-1}-\lambda_j^{r-1}}{\lambda_i-\lambda_j}\right].
\end{align}
This rearrangement is the useful step. The derivative of $u^{r-1}$ is nonpositive, so the second term inside the last brackets is nonnegative. If $L=\norm{X}_\op^2+\varepsilon$, then $0<\lambda_i,\lambda_j\le L$, and
\[
 \frac{\lambda_i^r-\lambda_j^r}{\lambda_i-\lambda_j}
 =r\int_0^1\bigl((1-s)\lambda_j+s\lambda_i\bigr)^{r-1}\dd s
 \ge rL^{r-1}.
\]
Consequently $\gamma_{ij}\ge r^2L^{r-1}$ for every $i,j$. Summing in \eqref{eq:exactcoeff} and using $\norm{C}_2=\norm{B}_2$ proves \eqref{eq:levilower}. Since $r-1\le0$, replacing $\norm{X}_\op$ by the larger number $R$ gives \eqref{eq:balllevi}.
\end{proof}

\begin{remark}
If $\lambda_i=\lambda_j=\lambda$, the coefficient in \eqref{eq:coefficient} equals
\[
 r\lambda^{r-2}\bigl(r\lambda+(1-r)\varepsilon\bigr)>0.
\]
Thus no division by zero occurs in the proof. If $s_i=s_j=0$, this coefficient is $r\varepsilon^{r-1}$; it may diverge as $\varepsilon\downarrow0$, but the lower bound remains valid. We shall pass to the limit in the function values and their circular means, not in the individual second derivatives.

For a nonzero scalar $a>0$, the real function $t\mapsto|a+t|^p$ has second derivative $p(p-1)a^{p-2}$, which is negative when $0<p<1$. In contrast,
\[
 \partial_z\partial_{\bar z}|a+z|^p\big|_{z=0}
 =\frac{p^2}{4}a^{p-2}>0.
\]
The complex Hessian combines the two real directions. Positivity of this combination is compatible with concavity along an individual real line.
\end{remark}

\section{Circular means and the isometric obstruction}\label{sec:circles}

The complex-curvature approach used here is rooted in the theory of uniform
plurisubharmonic convexity (PL-convexity) developed by Davis, Garling, and Tomczak-Jaegermann
\cite{DGT}. Blower and Ransford subsequently related this geometry to
Riesz measures of trace functions and established $2$-uniform
PL-convexity for Schatten trace ideals in the range $1\le p\le2$
\cite{BR}. In particular, in the trace-class case, for normalized
$A,B\in S_1$, the classical theory yields a circular-mean estimate of the form
\[
\frac{1}{2\pi}\int_0^{2\pi}
\|A+t e^{i\theta}B\|_1\,d\theta
\ge 1+c t^2,
\qquad 0<t<1,
\]
for some constant $c>0$. This quadratic growth along complex directions
provides the geometric starting point for the argument below. By combining
regularization with a trace-Hessian estimate, we retain an explicit
Hilbert--Schmidt contribution and extend the curvature mechanism to the
quasi-Banach range $0<p<1$.

\begin{theorem}\label{thm:circle}
Let $0<p\le2$, $A,B\in\Sp_p(\HH)$, $B\ne0$, and $t>0$. Then \eqref{eq:circleintro} holds. The constant $p^2/4$ cannot be increased in an inequality of this form that is required to hold for every $A,B,t$.
\end{theorem}

\begin{proof}
We first consider $n\times n$ matrices. Write
\[
 R=\norm{A}_\op+t\norm{B}_\op>0,
 \qquad c_\varepsilon=\frac{p^2}{4}(R^2+\varepsilon)^{p/2-1}\norm{B}_2^2.
\]
For $|z|\le t$, the operator norm triangle inequality gives $\norm{A+zB}_\op\le R$. Proposition~\ref{prop:levi} shows that
\[
 u_\varepsilon(z)=F_\varepsilon(A+zB)-c_\varepsilon|z|^2
\]
has nonnegative Laplacian in the open disc $|z|<t$. It is smooth on a neighborhood of the closed disc. The sub-mean inequality, or the elementary mean-value formula for the Laplacian, yields
\begin{equation}\label{eq:regmean}
 \mean{F_\varepsilon(A+t e^{i\theta}B)}
 \ge F_\varepsilon(A)+c_\varepsilon t^2.
\end{equation}
For completeness, a $C^2$ function $u$ with $\Delta u\ge0$ has increasing circular means: their derivative is $(2\pi s)^{-1}\int_{|z|<s}\Delta u\dd x\dd y\ge0$. Applying this observation to $u_\varepsilon$ proves the asserted inequality directly.

For $0<r\le1$ and $a\ge0$, scalar concavity gives
\[
 0\le(a+\varepsilon)^r-a^r\le\varepsilon^r.
\]
Applied to the eigenvalues of $X^*X$, this implies the uniform estimate
\[
 0\le F_\varepsilon(X)-\norm{X}_p^p\le n\varepsilon^{p/2}.
\]
Letting $\varepsilon\downarrow0$ in \eqref{eq:regmean} therefore proves \eqref{eq:circleintro} in finite dimensions, including singular $A$ and singular pencils.

Now suppose $\HH$ is infinite dimensional. It is enough to work on a separable reducing subspace containing the ranges of $A,B,A^*,B^*$. Choose finite-rank orthogonal projections $P_N$ increasing strongly to the identity on this subspace, and put
\[
 A_N=P_NAP_N,\qquad B_N=P_NBP_N.
\]
The standard ideal estimate $\norm{UXV}_p\le\norm{U}_\op\norm{X}_p\norm{V}_\op$ and density of finite-rank operators give
\begin{equation}\label{eq:compressconv}
 \norm{A_N-A}_p\longrightarrow0,
 \qquad \norm{B_N-B}_p\longrightarrow0.
\end{equation}
To see the density argument explicitly, first approximate $A$ by a fixed finite-rank $F$. The compressions of $A-F$ have quasi-norm at most $\norm{A-F}_p$. Moreover, $P_NFP_N\to F$ in operator norm and the ranks of their differences are bounded, so this convergence also holds in $\Sp_p$. The quasi-triangle inequality completes the argument, and the same reasoning applies to $B$.

The powers $\norm{A_N+zB_N}_p^p$ converge uniformly on $|z|\le t$. For $p\ge1$ this follows from the triangle inequality and continuity of $s\mapsto s^p$ on bounded intervals. For $p<1$, use the Schatten $p$-triangle inequality
\[
 \left|\norm{X}_p^p-\norm{Y}_p^p\right|\le\norm{X-Y}_p^p
\]
and \eqref{eq:compressconv}; these standard trace-ideal facts are discussed in \cite{McCarthy,Simon}. Also $\norm{B_N-B}_2\le\norm{B_N-B}_p\to0$.

The finite-dimensional proof can use the same upper bound $R$ for every compressed pencil, since
\[
 \norm{A_N+zB_N}_\op\le\norm{A}_\op+t\norm{B}_\op=R.
\]
It gives
\[
 \mean{\norm{A_N+t e^{i\theta}B_N}_p^p}-\norm{A_N}_p^p
 \ge\frac{p^2}{4}R^{p-2}t^2\norm{B_N}_2^2.
\]
Pass to the limit in this inequality. This proves \eqref{eq:circleintro} in infinite dimensions, without differentiating an infinite-dimensional trace function at a singular point.

Finally, take one-dimensional matrices $A=B=1$. For $t<1$, the binomial series gives
\[
 \mean{|1+t e^{i\theta}|^p}
 =\sum_{k=0}^\infty\binom{p/2}{k}^{\!2}t^{2k}
 =1+\frac{p^2}{4}t^2+O(t^4).
\]
The right-hand side of the proposed estimate with a constant $C$ in place of $p^2/4$ is $Ct^2(1+t)^{p-2}$. Dividing by $t^2$ and letting $t\downarrow0$ forces $C\le p^2/4$.
\end{proof}

\begin{corollary}\label{cor:flat}
Let $0<p\le2$, $0\ne A\in\Sp_p(\HH)$, and $B\in\Sp_p(\HH)$. If
\[
 \mean{\norm{A+t e^{i\theta}B}_p^p}-\norm{A}_p^p=o(t^2)
 \quad(t\downarrow0),
\]
then $B=0$.
\end{corollary}

\begin{proof}
If $B\ne0$, divide \eqref{eq:circleintro} by $t^2$ and take the lower limit. The result is at least $(p^2/4)\norm{A}_\op^{p-2}\norm{B}_2^2>0$, a contradiction.
\end{proof}

\begin{proof}[Proof of Theorem~\ref{thm:main}]
Suppose $T:\ell_q^2\to\Sp_p(\KK)$ is a complex-linear isometry, and put $A=T(1,0)$ and $B=T(0,1)$. Then $\norm{A}_p=\norm{B}_p=1$, and complex linearity gives, for every $\theta$,
\[
 \norm{A+t e^{i\theta}B}_p^p
 =\begin{cases}
 (1+t^q)^{p/q},&q<\infty,\\
 1,&q=\infty,\ 0<t<1.
 \end{cases}
\]
For finite $q>2$, the circular mean increment equals
\[
 \frac pq t^q+O(t^{2q})=o(t^2).
\]
For $q=\infty$ it is zero. In either case Corollary~\ref{cor:flat} implies $B=0$, contradicting $\norm{B}_p=1$.

If $P,Q$ are orthogonal rank-one projections on $\HH$, the map $(a,b)\mapsto aP+bQ$ is an isometry from $\ell_q^2$ into $\Sp_q(\HH)$, also for $q=\infty$. Restricting an alleged Schatten embedding to this span produces the prohibited map $T$. This proves all the assertions.
\end{proof}

\begin{remark}
The identity involving $e^{i\theta}B$ is the only place where the source is used, but it is decisive. A merely real-linear map from $\ell_q^2(\R)$ need not preserve this identity. No conclusion about that different embedding problem is inferred from Theorem~\ref{thm:main}.
\end{remark}

\label{sec:distortion}

The obstruction has a quantitative form when an image direction has bounded rank. This is especially relevant for quasi-Banach targets, where Hilbert--Schmidt control need not give a dimension-independent lower bound in terms of the Schatten $p$ quasi-norm.

\begin{definition}
For an injective complex-linear map $T:\ell_q^2\to\Sp_p(\HH)$, define
\[
 M(T)=\sup_{x\ne0}\frac{\norm{Tx}_p}{\norm{x}_q},\qquad
 m(T)=\inf_{x\ne0}\frac{\norm{Tx}_p}{\norm{x}_q},\qquad
 D(T)=\frac{M(T)}{m(T)}.
\]
The finite-dimensional domain and injectivity imply $0<m(T)\le M(T)<\infty$. These definitions remain valid for $p<1$.
\end{definition}

For $2<q<\infty$, set $\Phi_{p,q}(t)=(1+t^q)^{p/q}$; for $q=\infty$, set $\Phi_{p,\infty}(t)=1$ on $0\le t\le1$.

\begin{theorem}\label{thm:distortion}
Let $0<p\le2$, $2<q\le\infty$, and let $T:\ell_q^2\to\Sp_p(\HH)$ be injective and complex-linear. Suppose that $\rank T(0,1)\le k<\infty$. Then
\begin{equation}\label{eq:distortion}
 D(T)^2\ge
 \sup_{0<t\le1}
 \frac{1+\dfrac{p^2}{4}k^{1-2/p}t^2(1+t)^{p-2}}
 {\Phi_{p,q}(t)}
 >1.
\end{equation}
In particular, the assertion applies with $k=n$ to every injective complex-linear map into $\Sp_p^n$.
\end{theorem}

\begin{proof}
Multiplying $T$ by a positive scalar, we may assume $M(T)=1$. Write $D=D(T)$, $A=T(1,0)$, and $B=T(0,1)$. We have
\[
 D^{-1}\le\norm{A}_p,\norm{B}_p\le1,
 \qquad \norm{A}_\op,\norm{B}_\op\le1.
\]
The power-mean inequality for the at most $k$ nonzero singular values of $B$ gives
\begin{equation}\label{eq:rankcomparison}
 \norm{B}_2^2\ge k^{1-2/p}\norm{B}_p^2
 \ge k^{1-2/p}D^{-2}.
\end{equation}
The circular mean on the left of \eqref{eq:circleintro}, before subtracting the central value, is at most $\Phi_{p,q}(t)$. Since $p-2\le0$ and $\norm{A}_\op+t\norm{B}_\op\le1+t$, Theorem~\ref{thm:circle} implies
\[
 \Phi_{p,q}(t)\ge D^{-p}
 +\frac{p^2}{4}(1+t)^{p-2}t^2k^{1-2/p}D^{-2}.
\]
Multiply by $D^2$ and use $D^{2-p}\ge1$. This proves the non-strict inequality in \eqref{eq:distortion}.

If $q<\infty$, the numerator is $1+(p^2/4)k^{1-2/p}t^2+O(t^3)$, whereas the denominator is $1+O(t^q)$. Their ratio is greater than one for sufficiently small $t>0$. If $q=\infty$, the denominator is one and every $t>0$ gives strict inequality.
\end{proof}

\begin{corollary}\label{cor:gap}
Put $a=p^2 2^{p-4}k^{1-2/p}$. For $2<q<\infty$, define
\[
 \tau=\min\left\{1,\left(\frac{aq}{2p}\right)^{1/(q-2)}\right\}.
\]
Under the hypotheses of Theorem~\ref{thm:distortion},
\[
 D(T)\ge\sqrt{1+\frac{a\tau^2}{3}}>1.
\]
For $q=\infty$, one has $D(T)\ge\sqrt{1+a}$.
\end{corollary}

\begin{proof}
For $0<t\le1$, $(1+t)^{p-2}\ge2^{p-2}$, so the numerator in \eqref{eq:distortion} is at least $1+at^2$. Since $0<p/q<1$, concavity gives
\[
 (1+t^q)^{p/q}\le1+(p/q)t^q.
\]
The definition of $\tau$ ensures $(p/q)\tau^q\le a\tau^2/2$. Also $0<a\le1$. Hence
\[
 D(T)^2\ge\frac{1+a\tau^2}{1+a\tau^2/2}
 =1+\frac{a\tau^2}{2+a\tau^2}
 \ge1+\frac{a\tau^2}{3}.
\]
When $q=\infty$, take $t=1$ directly in \eqref{eq:distortion}.
\end{proof}

\begin{proposition}\label{prop:hilbert}
For $2<q\le\infty$, the least distortion of a complex-linear embedding of $\ell_q^2$ into a Hilbert space of dimension at least two is
\[
 2^{1/2-1/q},
\]
with $1/\infty=0$. Theorem~\ref{thm:distortion} recovers this exact value when $p=2$.
\end{proposition}

\begin{proof}
When $p=2$, the rank term in \eqref{eq:distortion} equals one. Choosing $t=1$ gives
\[
 D(T)^2\ge\frac{2}{2^{2/q}}.
\]
This lower bound does not require a rank assumption: for a Hilbert target it follows directly from the identity
\[
 \mean{\norm{A+t e^{i\theta}B}_2^2}
 =\norm{A}_2^2+t^2\norm{B}_2^2
\]
and the normalization used in the preceding proof. Conversely, the identity map from $\C^2$ with the $\ell_q$ norm to $\C^2$ with the Euclidean norm satisfies
\[
 \norm{x}_q\le\norm{x}_2\le2^{1/2-1/q}\norm{x}_q,
\]
and both constants are attained. This gives the matching upper bound.
\end{proof}

For $p<2$, we do not claim the rank-dependent lower bound is optimal. Its purpose is to turn an exact obstruction into an explicit separation from isometry for each fixed rank. In particular, it excludes a sequence of distortions tending to one when all image directions under consideration have a common finite rank bound. At $p=1$, classical uniform complex convexity gives stronger, dimension-independent information.

\section{Circular expansions at a rank drop}\label{sec:expansion}

The circular lower bound does not by itself exclude $\ell_1^2$ from a target with $p<1$: linear growth is compatible with positive quadratic growth. Finite-dimensional holomorphic pencils possess additional structure. Their singular values vanish to integer orders, and angular averaging removes the odd corrections within each group of a fixed vanishing order.

\begin{lemma}\label{lem:smith}
Let $M(z)$ be a holomorphic $n\times n$ matrix near zero, of rank $s$ over the field of meromorphic germs. There are integers $0\le k_1\le\cdots\le k_s$ and holomorphic invertible matrices $E(z),G(z)$ near zero such that
\begin{equation}\label{eq:smith}
 M(z)=E(z)\operatorname{diag}(z^{k_1},\ldots,z^{k_s},0,\ldots,0)G(z).
\end{equation}
In a sufficiently small punctured disc, the nonzero singular values, arranged decreasingly, satisfy
\begin{equation}\label{eq:orders}
 c|z|^{k_j}\le s_j(M(z))\le C|z|^{k_j},\qquad 1\le j\le s,
\end{equation}
for constants $c,C>0$ independent of $\arg z$.
\end{lemma}

\begin{proof}
{
The factorization \eqref{eq:smith} is the standard local Smith form for a holomorphic matrix pencil; see, for example, \cite{KM}. We therefore only record the consequence needed below. Shrinking the disc if necessary, the matrices $E,G,E^{-1},G^{-1}$ are uniformly bounded in operator norm. Hence the standard singular-value inequalities for multiplication by bounded invertible matrices compare the nonzero singular values of $M(z)$ with those of
\[
 \operatorname{diag}(z^{k_1},\ldots,z^{k_s},0,\ldots,0),
\]
uniformly for small $z$. Since $0\le k_1\le\cdots\le k_s$, this gives \eqref{eq:orders} for $|z|<1$.
}
\end{proof}

\begin{lemma}\label{lem:radial}
Let $M(z)$ be a holomorphic matrix near zero and let $p>0$. For sufficiently small $t>0$,
\begin{equation}\label{eq:radialdecomp}
 \mean{\Tr|M(t e^{i\theta})|^p}
 =\sum_{k\in J}t^{kp}G_k(t),
\end{equation}
where $J$ is the finite set of vanishing orders occurring in Lemma~\ref{lem:smith}. Each $G_k$ extends to an even real-analytic function near $t=0$ and satisfies $G_k(0)>0$. If $M$ is identically zero, the sum is empty.
\end{lemma}

\begin{proof}
Set $H(t,\theta)=M(t e^{i\theta})^*M(t e^{i\theta})$. Its entries are real-analytic in the real variables $(t,\theta)$, including negative $t$. Its nonzero eigenvalues of order $k$ are uniformly comparable to $|t|^{2k}$, by \eqref{eq:orders}. We describe carefully why the eigenvalues can be grouped analytically without choosing analytic individual eigenvectors at a crossing.

First separate the eigenvalues of $H(0,\theta)$ that are positive from zero. On a neighborhood of any fixed $\theta_0$, the corresponding spectral projection is real-analytic in $(t,\theta)$: it is the contour integral of $(\zeta I-H(t,\theta))^{-1}$ around that positive spectral cluster. A real-analytic orthonormal frame for its range is available locally. Indeed, if $W$ is a fixed orthonormal frame at $(0,\theta_0)$ and $P$ is the varying projection, the columns of
\[
 PW(W^*PW)^{-1/2}
\]
form such a frame nearby. The same construction on the complementary projection gives a real-analytic unitary block reduction of $H$ into a positive block $K_0$ and a residual block $R_0$.

The eigenvalues of $R_0$ are the remaining eigenvalues of $H$. Their largest is $O(t^2)$, uniformly on a smaller angular neighborhood. As $R_0$ is positive semidefinite for real $t$, its operator norm is $O(t^2)$; consequently every entry and its first $t$-derivative vanish at $t=0$. Real analyticity implies an exact factorization $R_0=t^2H_1$ with real-analytic $H_1$. By continuity, $H_1$ remains positive semidefinite at $t=0$.

The eigenvalues of $H_1$ associated with $k_j=1$ are bounded above and below by positive constants as $t\to0$. Those with $k_j>1$ tend to zero. Thus its positive spectral cluster at $t=0$ has a fixed rank, separated uniformly from zero; repeat the same block reduction. After removing the positive block $K_1$, the residual has norm $O(t^2)$ and is again divisible by $t^2$.

Continuing in this way, and skipping empty positive clusters, yields for each $k\in J$ a positive definite real-analytic block $K_k(t,\theta)$. The eigenvalues of $t^{2k}K_k(t,\theta)$ are exactly the eigenvalues of $H$ of order $2k$. After $\max J$ stages, all remaining eigenvalues are identically zero. This last assertion follows from \eqref{eq:smith}, which fixes the rank on the punctured disc; a Hermitian residual with all eigenvalues zero is the zero matrix. The construction also applies if the initial positive cluster is empty.

Put $g_k(t,\theta)=\Tr K_k(t,\theta)^{p/2}$. Positive definiteness makes $g_k$ real-analytic, and $g_k(0,\theta)>0$. The intervals $[c^2|t|^{2k},C^2|t|^{2k}]$ for distinct $k$ are disjoint after decreasing $\delta$. Although frames were constructed locally, the scalar functions therefore agree on overlaps: for $t\ne0$, they are the sum of the $p/2$ powers of the eigenvalues in the uniquely specified order-$2k$ cluster, divided by $|t|^{kp}$. Equality at zero follows by continuity. Compactness of the angular circle gives a common interval $|t|<\delta$, on which the functions and their analytic expansions are uniform in $\theta$. It follows that
\[
 G_k(t)=\mean{g_k(t,\theta)}
\]
is real-analytic. Equivalently, this follows by integrating their locally convergent Taylor series, whose convergence radii can be chosen uniformly using a finite cover of the circle.

Finally, $H(-t,\theta)=H(t,\theta+\pi)$. Since this identity preserves each order cluster, it gives
\[
 g_k(-t,\theta)=g_k(t,\theta+\pi).
\]
Integration in $\theta$ proves $G_k(-t)=G_k(t)$. The exact decomposition of the eigenvalues gives \eqref{eq:radialdecomp}, and positivity gives $G_k(0)>0$.
\end{proof}

\begin{proposition}\label{prop:loworder}
Let $p>0$, $A\ne0$, and $B\in M_n(\C)$. There are a finite set $J_+\subset\N$, positive numbers $a_k$ for $k\in J_+$, and $b\ge0$ such that
\begin{equation}\label{eq:loworder}
 \mean{\norm{A+t e^{i\theta}B}_p^p}
 =\norm{A}_p^p+\sum_{\substack{k\in J_+\\kp\le2}}a_k t^{kp}
 +bt^2+o(t^2).
\end{equation}
Moreover, in singular-value coordinates $A=\operatorname{diag}(s_1,\ldots,s_h,0,\ldots,0)$ with $s_i>0$, the coefficient $b$ vanishes if and only if $B$ is supported entirely in the bottom-right $(n-h)\times(n-h)$ block.
\end{proposition}

\begin{proof}
Apply Lemma~\ref{lem:radial} to $M(z)=A+zB$. For $k\ge1$, even analyticity gives $G_k(t)=a_k+O(t^2)$ with $a_k>0$. The correction $t^{kp}O(t^2)$ is $o(t^2)$, and terms with $kp>2$ can also be absorbed into the remainder. The order-zero group gives
\[
 G_0(t)=\norm{A}_p^p+bt^2+O(t^4).
\]
It remains to determine the sign and zero set of $b$.

Let $v(z)$ be the sum of the $p/2$ powers of the eigenvalues of $(A+zB)^*(A+zB)$ that are positive at $z=0$. This is a real-analytic function near zero, and $b=\partial_z\partial_{\bar z}v(0)$. Put $\alpha=p/2$ and $\lambda_i=s_i^2$ for $i\le h$. Choose a scalar function $f$ analytic on disjoint neighborhoods of the positive spectrum and of zero, equal to $u^\alpha$ on the former and to zero on the latter. Then $v(z)=\Tr f((A+zB)^*(A+zB))$ near zero. Lemma~\ref{lem:trace} applies and yields
\begin{align}\label{eq:regularcoefficient}
 b={}&\alpha\sum_{i,j\le h}
 \frac{\lambda_i^\alpha-\lambda_j^\alpha}{\lambda_i-\lambda_j}|b_{ij}|^2\notag\\
 &+\alpha\sum_{i\le h<j}s_i^{p-2}
 \bigl(|b_{ij}|^2+|b_{ji}|^2\bigr),
\end{align}
where the quotient on the diagonal is $\alpha\lambda_i^{\alpha-1}$. Every coefficient displayed is strictly positive for every $p>0$. The kernel-to-kernel entries have coefficient zero, since $f$ is identically zero near zero. This proves $b\ge0$ and the assertion about equality. Two-sided unitary changes of coordinates preserve the argument.
\end{proof}

\begin{remark}
The conclusion is more precise than subharmonicity. Terms $t^{kp+1}$ can occur before integration in $\theta$, but their angular integrals vanish. All corrections to a leading singular contribution then have order at least $kp+2>2$. Therefore, below and at order two, only the positive leading singular coefficients and the nonnegative regular curvature survive. No positivity assertion is made for higher-order coefficients.
\end{remark}

\begin{proof}[Proof of Theorem~\ref{thm:finiteintro}]
For $q=p$, the diagonal map gives the required isometry. For $q=2$, use a matrix whose first row is $(a,b,0,\ldots,0)$ and whose other rows vanish. These are the positive cases.

Conversely, suppose $T:\ell_q^2\to\Sp_p^n$ is an isometry and write $A=T(1,0)$, $B=T(0,1)$. Both images have Schatten $p$ quasi-norm one. If $q<\infty$, the circular mean is
\begin{equation}\label{eq:sourceexpansion}
 (1+t^q)^{p/q}
 =1+\frac pq t^q+\frac{p(p-q)}{2q^2}t^{2q}+O(t^{3q}).
\end{equation}
For $q=\infty$, it is one when $0<t<1$.

Suppose first that $1<q<\infty$ and $q\ne2$. If $q<2$, the expansion through order two of \eqref{eq:sourceexpansion} is $1+(p/q)t^q+o(t^2)$. If $q>2$, it is $1+o(t^2)$. In both cases the coefficient of $t^2$ is zero. Comparing with \eqref{eq:loworder}, whose displayed coefficients are all nonnegative, shows that
\[
 b+\sum_{\substack{k\in J_+\\kp=2}}a_k=0,
\]
and hence $b=0$. This comparison uses only uniqueness of a finite expansion in distinct real powers: subtract the two expansions, divide by the smallest power with a nonzero coefficient, and let $t\downarrow0$; repeat for the remaining powers.

By Proposition~\ref{prop:loworder}, $B$ is supported on the complementary left and right kernels of $A$. Thus the pencil is block diagonal in singular-value coordinates, and
\[
 \norm{A+zB}_p^p=\norm{A}_p^p+|z|^p\norm{B}_p^p=1+|z|^p.
\]
Its equality with $(1+|z|^q)^{p/q}$ forces $p=q$, by comparison of the first nonconstant powers at zero. The same reasoning for $q=\infty$ gives $b=0$ and then contradicts the asserted constant mean.

It remains to consider $0<q\le1$. If $p>q$, every exponent $kp$ in \eqref{eq:loworder} is larger than $q$, as is the regular exponent two. The circular mean increment is then $o(t^q)$, contradicting \eqref{eq:sourceexpansion}. If $p<q$, the coefficient of $t^{2q}$ in \eqref{eq:sourceexpansion} is strictly negative and $2q\le2$. The expansion \eqref{eq:loworder}, truncated at order $2q$, has only nonnegative coefficients; uniqueness of finite real-power expansions again gives a contradiction. Thus $p=q$ in this final case as well. Restriction to two diagonal matrix units proves the assertion for Schatten source spaces.
\end{proof}

In particular, at $q=1$ and $0<p<1$ the decisive contradiction is
\[
 (1+t)^p=1+pt-\frac{p(1-p)}2t^2+o(t^2).
\]
If no vanishing order satisfies $kp=1$, the linear term is already impossible. At the exceptional exponents $p=1/k$, that term is allowed, but its next possible circular correction has order three, not two. The negative quadratic coefficient therefore remains impossible. This directly addresses the reciprocal-exponent obstruction left by a real-parameter expansion.

\section{Consequences for Problem 3.1 in \cite{CPS}}

The applications below refer to the numbering in \cite{CPS}, version~2. They are consequences of the theorems proved here, not assertions that all the listed questions have been resolved.

{ A closely related two-dimensional phenomenon was established by Hein\"avaara \cite{Heinavaara2024}, who proved that every two-dimensional real subspace of $S_3$ is linearly isometric to a subspace of the commutative space $L_3$. Whereas that result concerns the commutative representation of planes contained in $S_3$, the problem considered here concerns complex-linear isometric embeddings of finite-dimensional Schatten-$3$ spaces into trace-class targets. Our finite-dimensional classification rules out such embeddings; in particular, $\Sp_3^m\not\hookrightarrow\Sp_1^n$ for $m\ge2$, giving a negative answer to the corresponding case in \cite{CPS}.}

\begin{corollary}\label{cor:answers}
Let $m\ge2$ and $n<\infty$.
\begin{enumerate}
\item There is no complex-linear isometric embedding $\Sp_3^m\to\Sp_1^n$.
\item For every $q>2$ and $0<p<1$, there is no complex-linear isometric embedding $\Sp_q^m\to\Sp_p^n$.
\item For every $0<p<1$, including every $p=1/k$ with integer $k\ge2$, neither $\Sp_1^m$ nor $\Sp_\infty^m$ admits a complex-linear isometric embedding into $\Sp_p^n$.
\item If $\dim\HH\ge2$, then $\Sp_q(\HH)$ does not embed complex-linearly isometrically into $\Sp_p(\KK)$ for $2<q\le\infty$ and $0<p<1$, or for $2<q<4$ and $p=1$.
\end{enumerate}
\end{corollary}

\begin{proof}
Apply Theorem~\ref{thm:finiteintro} for the finite-dimensional assertions and Theorem~\ref{thm:main} for the infinite-dimensional assertion. A Schatten source contains the diagonal span of two orthogonal rank-one projections; neither $m\le n$ nor any additional structural hypothesis on the map is needed for the negative conclusion.
\end{proof}

The results obtained in this paper provide several consequences for the Schatten
embedding problems considered in \cite{CPS}. In particular, they give complete
negative answers for the cases $\Sp_3^m\not\hookrightarrow\Sp_1^n$, the
reciprocal-exponent obstructions for $\Sp_1^m$ and $\Sp_\infty^m$, and the whole
range $q>2$, $0<p<1$. Combining complex curvature estimates with
finite-dimensional holomorphic expansions, we obtain a complete classification
of the isometric embeddings $\ell_q^2(\mathbb C)\hookrightarrow\Sp_p^n$.
The trace-class case is consistent with classical complex convexity theory,
while the Hilbertian dimension problem and the remaining infinite-dimensional
parameter ranges are beyond the scope of the present work.

	\section*{Acknowledgements}
%====================================================================
Thanks to all the members of the Functional Analysis Research team of the College of
Mathematics and Statistics of Anqing Normal University for their discussion and correction of the difficulties and errors encountered in this paper.

\end{document}